\documentclass[a4paper,11pt]{article}
\usepackage{amsthm,amsmath,amssymb,amsfonts}

\newtheorem{theorem}{Theorem}
\newtheorem{proposition}{Proposition}
\newtheorem{lemma}{Lemma}
\newtheorem*{fact}{Fact}
\theoremstyle{remark}
\newtheorem*{observation*}{Observation}
\newtheorem*{claim}{Claim}

\newcommand*{\contract}{\star}
\newcommand*{\transp}{\mathsf{T}}

\newcommand*{\N}{\mathbb{Z}_{+}}
\newcommand*{\Z}{\mathbb{Z}}

\newcommand*{\R}{\mathbb{R}}
\newcommand*{\Rplus}{\mathbb{R}_{+}}
\newcommand*{\Q}{\mathbb{Q}}
\newcommand*{\supp}{\mathrm{supp}}
\newcommand*{\cl}{\mathrm{cl}}
\newcommand*{\join}{\vee}
\newcommand*{\meet}{\wedge}
\newcommand*{\one}{\mathbf{1}}

\title{An algorithm for weighted fractional matroid matching}
\author{Dion Gijswijt \thanks{CWI and Department of Mathematics, Leiden University. E-mail: {\tt dion.gijswijt@gmail.com}} \and Gyula Pap  \thanks{MTA-ELTE Egerv\'ary Research Group on Combinatorial Optimization, Department of Operations Research, E\"otv\"os Lor\'and University, P\'azm\'any P\'eter s\'et\'any, 1/C,
1117 Budapest, Hungary. E-mail: \texttt{gyuszko@cs.elte.hu}. Supported by the Hungarian National Foundation for Scientific Research (OTKA) grant CK80124.}}

\begin{document}

\maketitle

\begin{abstract}
Let $M$ be a matroid on ground set $E$ with rank function $r$. A subset $l\subseteq E$ is called a \emph{line} when $r(l)\in \{1,2\}$. Given a finite set $L$ of lines in $M$, a vector $x\in \Rplus^L$ is called a \emph{fractional matching} when $\sum_{l\in L} x_l a(F)_l\leq r(F)$ for every flat $F$ of $M$. Here $a(F)_l$ is equal to $0$ when $l\cap F=\emptyset$, equal to $2$ when $l\subseteq F$ and equal to $1$ otherwise. We refer to $\sum_{l\in L}x_l$ as the \emph{size} of $x$.

It was shown by Chang et al. (\textit{Discrete Math.} 237 (2001), 29--61.), that a maximum size fractional matching can be found in polynomial time. In this paper we give a polynomial time algorithm to find, for given weight function $w:L\to \Q$, a \emph{maximum weight} fractional matching. 
A simple reference to the equivalence of separation and optimization does not lead to such an algorithm, since no direct method for polynomial time separation is known for this polytope. 
\end{abstract}

\section{Introduction}
Let $M$ be a matroid on the set $E$, and let $\{a_1,b_1\},\ldots,\{a_k,b_k\}$ be pairs of elements from $E$. A subset $I$ of $\{1,\ldots,k\}$ is called a \emph{matching} if $\bigcup_{i\in I} \{a_i,b_i\}$ is an independent set of size $2|I|$ in the matroid $M$. The problem of finding a maximum size matching is called the \emph{matroid matching problem} and was proposed by Lawler \cite{Lawler} as a common generalization of non-bipartite matching and matroid intersection. By adding parallel copies of elements in the matroid, it may be assumed that the pairs are disjoint. In this form the matroid matching problem is often called the \emph{matroid parity problem}.

Lov\'asz \cite{Lovasz} and, independently, Korte and Jensen \cite{KorteJensen} showed that for general matroids given by an independence oracle, the matroid matching problem requires an exponential number of oracle calls, and hence is not solvable in polynomial time. Also, the problem of finding a maximum size clique in a graph can be formulated as a matroid matching problem, where the oracle is removed and independence can be read off from the input graph directly, showing that the matroid matching problem contains NP-hard problems (see \cite{Lex}).

However, matroid matching has become a powerful tool in combinatorial optimization since Lov\'asz \cite{Lovasz} proved a min-max formula and constructed a polynomial time algorithm for matroid matching in representable matroids that are given by an explicit linear representation over a field. Already linear matroid matching has a wide range of applications, among which are a polynomial time algorithm for packing Mader paths \cite{Lex}, graph rigidity \cite{Lovasz2} and computing the maximum genus surface in which an input graph can be embedded with all faces isomorphic to the disk \cite{genus}.
More efficient algorithms for linear matroid matching have been developed by Gabow and Stallmann \cite{GabowStallmann} and by Orlin and Vande Vate \cite{OrlinVandeVate}.     

An outstanding problem concerning matroid matching is to construct an efficient algorithm for finding a \emph{maximum weight} matching for linearly represented matroids\footnote{At time of publication we note that such an algorithm has recently been found by Satoru Iwata and independently by Gyula Pap. By comparison, the present paper is not limited to linearly representable matroids, but rather considers a fractional relaxation of the problem for general matroids.}.  

To gain better understanding of the matroid matching polytope, Vande Vate \cite{FMM} introduced a fractional relaxation of this polytope, called the \emph{fractional matroid matching polytope}. The integer points of this polytope correspond exactly to the matroid matchings, and although the polytope itself is not integer, the vertices are half-integer. Furthermore, in the two extremes where the matroid matching problem is in fact a matroid intersection problem or a non-bipartite matching problem, the fractional matroid matching polytope coincides with the common independent set polytope and with the fractional matching polytope respectively.

In a series of two papers, Chang, Llewellyn and Vande Vate \cite{FMM-algorithm, FMM-duality} showed that there exists a polynomial time algorithm to optimize the all-one objective function over the fractional matroid matching polytope. 

In the present paper, we consider optimizing arbitrary weight functions over the fractional matroid matching polytope. Extending the result from \cite{FMM}, we show that not only is the polytope half-integer, the system defining the polytope is in fact totally dual half-integral. Our main result is the construction of a polynomial time algorithm for optimizing arbitrary weight functions over the fractional matroid matching polytope, showing that the polytope is algorithmically tractable. 
Note that this result does not follow by reference to the equivalence of separation and optimization, since separation of the polytope is not straightforward. In fact, no direct method for polynomial time separation is known for this polytope. 

%
%
%

\section{Preliminaries}
Let $M$ be a matroid with ground set $E$ and rank function $r_M$. Here we will assume that $M$ does not have loops. The set $E$ need not be finite, but we do require that $M$ has finite rank. The span of $X\subseteq E$ is denoted by $\cl_M (X)$, the smallest flat containing $X$. The set $\mathcal{L}(M)$ of flats of the matroid $M$ is a lattice under inclusion, with join $S\join T:=\cl_M (S\cup T)$ and meet $S\meet T:=S\cap T$. The rank function $r_M$ is submodular on $\mathcal{L}(M)$:
\begin{equation}r_M(S)+r_M(T)\geq r_M(S\join T)+r_M(S\meet T)\quad \text{for all $S,T\in \mathcal{L}(M)$.}
\end{equation}
When the matroid is clear from the context, we will suppress the matroid in the notation. For further notation and preliminaries on matroids, we refer the reader to \cite{Oxley}. 

Let $L$ be a finite set of \emph{lines} of $M$, where a \emph{line} is a subset of $E$ of rank $1$ or $2$. We do not require the lines to be flats. For any subset $X\subseteq E$, we define its \emph{degree vector} $a(X)\in \{0,1,2\}^L$ by 
\begin{equation}\label{Degreevector}
a(X)_l:=
\begin{cases}
 0 & \text{if $X\cap l=\emptyset$,}\\
 2 & \text{if $X\supseteq l$,}\\
 1 & \text{otherwise.}
\end{cases}
\end{equation}

The following important fact will be used throughout the text.
\begin{fact}
For every line $l\in L$ the function $a(\cdot)_l:\mathcal{L}(M)\to \{0,1,2\}$ is \emph{supermodular}: 
\begin{equation}
a(S)_l+a(T)_l\leq a(S\join T)_l+a(S\meet T)_l\quad\text{for all $l\in L$ and all $S,T\in\mathcal{L}(M)$.}
\end{equation}
\end{fact}
\begin{proof}
We may assume without loss of generality that $a(S)_l\leq a(T)_l$. Then there are three cases to consider.
\begin{itemize}
\item If $a(S)_l=0$, the inequality follows since $a(T)_l\leq a(S\vee T)_l$.
\item If $a(T)_l=2$, then $l\subseteq T$ which implies that $l\cap S=l\cap (S\wedge T)$ and hence $a(S)_l=a(S\wedge T)_l$. The inequality now follows since $a(T)_l\leq a(S\vee T)_l$.
\item If $a(S)_l=a(T)_l=1$, then either $l\cap (S\cap T)\neq \emptyset$ or $l\cap (S\setminus T), l\cap (T\setminus S)\neq \emptyset$.\\
In the first case we have  $a(S\wedge T)_l\geq 1$ and hence also $a(S\vee T)_l\geq 1$.\\
In the second case we find that $l\subseteq S\vee T$ and hence $a(S\vee T)_l=2$.\\In both cases the inequality follows directly.
\end{itemize}
\end{proof}

A vector $x\in \R^L$ is called a \emph{fractional matching in $(M,L)$} when 
\begin{eqnarray}\label{FMpolytope}
x&\geq &0,\nonumber\\
a(T)\cdot x&\leq&r(T) \text{ for every $T\in\mathcal{L}(M)$}.
\end{eqnarray}
Observe that for any line $l\in L$ we have $x_l\leq 1$ if $r(l)=2$ and $x_l\leq \tfrac{1}{2}$ if $r(l)=1$ by taking $T=\cl(l)$.
The set of fractional matchings in $(M,L)$ is called the \emph{fractional matching polytope}, denoted by $P_{M,L}$.  Although the number of flats may be infinite, the number of distinct degree-vectors $a(T)$ is finite so that $P_{M,L}$ is indeed a polytope.

We refer to $|x|:=\sum_{l\in L}x_l$ as the \emph{size} of $x$. The maximum size of a fractional matching in $(M,L)$ is denoted by $\nu^*(M,L)$. The size of a fractional matching $x$ is at most $\frac{1}{2}r(E)$, and when equality holds, $x$ is called a \emph{perfect} fractional matching.

A simple but very useful observation is the following.  
\begin{equation}
\text{For $x\in P_{M,L}$, $\{T\in \mathcal{L}(M)\mid a(T)\cdot x=r(T)\}$ is a sublattice of $\mathcal{L}(M)$}.
\end{equation}
This follows directly from supermodularity of the $a_l(\cdot)$ and submodularity of the rank function (see \cite{FMM-duality}).  

In what follows, it will be convenient to work with finite formal sums of flats, that is, expressions of the form $y=\sum_{F\in \mathcal{L}(M)} \lambda_F F$, where $\lambda\in \R^{\mathcal{L}(M)}$ and $\lambda_F=0$ for all but a finite number of flats $F$. We denote $y_F:=\lambda_F$ and $\supp(y):=\{F\mid y_F\not=0\}$. We extend all functions on $\mathcal{L}(M)$ linearly to finite formal sums and hence write $r(y):=\sum_F \lambda_F r(F)$ and $a(y):=\sum_F \lambda_F a(F)$.

Consider the linear program of optimizing a nonnegative weight function $w:L\to\N$ over the fractional matching polytope:

\begin{eqnarray}\label{primal}
\text{maximize}&&w\cdot x\nonumber\\
\text{subject to}&&x\geq 0\nonumber\\
&&a(T)\cdot x\leq r(T)\quad \text{for all $T\in\mathcal{L}(M)$}.
\end{eqnarray}

The dual linear program is given over vectors $y\in\R ^{\mathcal{L}(M)}$ of finite support\footnote{For each of the finitely many degree vectors, we may restrict to at most one representing flat $F$.} by
\begin{eqnarray}\label{dual}
\text{minimize}&& r(y)\nonumber\\
\text{subject to}&&y\geq 0\nonumber\\
&&a(y)\geq w.
\end{eqnarray}

It was shown in \cite{FMM} that the fractional matching polytope is half-integer and that for $w=\one$, there is a half-integer optimal dual solution. Here we show that the system (\ref{primal}) is totally dual half integral. That is, for every $w:L\to \N$, the dual has a half-integer optimal dual solution. This implies half-integrality of the fractional matching polytope (see \cite{Linearprogramming}).

\begin{theorem}
System (\ref{primal}), describing the fractional matching polytope, is totally dual half-integral.
\end{theorem}
\begin{proof}
We show that the dual linear program (\ref{dual}) has a half-integer solution. Let $y$ be an optimal solution for which $f(y):=\sum_{F\in \mathcal{L}(M)} y_F r(F)^2$ is maximal. Such a $y$ exists, because we may restrict to a finite set of flats, one representative for each possible degree-vector, and because the function $f(y)$ is bounded from above by $r(E)\cdot r(y)$. 
 
The support $\mathcal{L}^+:=\supp (y)$ of $y$ is a chain. Indeed, suppose that $S,T\in\mathcal{L}^+$ with $r(T)\geq r(S)$ and $S\not\subseteq T$. Let $\epsilon:=\min \{y_S,y_T\}>0$, and define $y':=y+\epsilon\,(S\meet T+S\join T-S-T)$. Then $y'$ is a feasible solution to (\ref{dual}) by supermodularity of the $a(\cdot)_l$, $l\in L$. By the submodularity of $r$, $y'$ will be an optimal dual solution and $d:=r(S\join T)-r(T)=r(S)-r(S\meet T)>0$. It then follows that $f(y')-f(y)=2\epsilon d\cdot(r(T)-r(S)+d)\geq 2\epsilon d\cdot d>0$, contradicting the choice of $y$.

Denote $\mathcal{L}^+=\{T_1,\ldots, T_k\}$ where $T_1\subset \cdots\subset T_k$. Let $A\in \{0,1,2\}^{\{1,2,\ldots,k\}\times L}$ given by $A_{i,l}:=a(T_i)_l$, be the matrix with the degree vectors of the $T_i$ as rows. To prove that there exists a half-integral optimum $y$, it suffices to show that every non-singular square submatrix $A'$ of $A$ has a half-integer inverse. Since $A'^{-1}=(UA')^{-1}U$, where 
$$U=
\begin{bmatrix}
1&0&\cdots&\cdots&0\\
-1&\ddots&\ddots&&\vdots\\
0&\ddots&\ddots&\ddots&\vdots\\
\vdots&\ddots&\ddots&\ddots&0\\
0&\hdots&0&-1&1
\end{bmatrix},$$
it suffices to show that $D:=UA'$ has a half-integer inverse. This follows from the following claim.
\begin{claim}
Let $D\in \{0,1,2\}^{n\times n}$ be a nonsingular matrix with column sums at most $2$, and let $b\in \Z^n$. Then the unique solution $x$ of $Dx=b$ is half-integer.  
\end{claim}


\begin{proof}[Proof of claim.]
The proof is by induction on $n$, the case $n=1$ being clear. First suppose that some row $i$ of $D$ has sum $1$. By exchanging rows and columns, we may assume that our system has the form
$$\begin{bmatrix}1&0\\d'&D'\end{bmatrix}\,\begin{bmatrix}x_1\\x'\end{bmatrix}=\begin{bmatrix}b_1\\b'\end{bmatrix}.$$
Hence $x_1=b_1$ and $[x_2,\ldots, x_n]^\transp$ is the unique solution to $D'x'=b'-b_1d'\in \mathbb{Z}^{n-1}$, which by induction has a half-integer solution.

So we may assume that all row-sums of $D$ are at least 2. Since $D$ is square and all column-sums are at most 2, a simple counting argument shows that all row- and column-sums are equal to 2.

If $D_{ij}=2$ for some $i$ and $j$, we may assume that $i=j=1$ and obtain the system
$$
\begin{bmatrix}2&0\\0&D'\end{bmatrix}\begin{bmatrix}x_1&x'\end{bmatrix}=\begin{bmatrix}b_1\\b'\end{bmatrix}.
$$
We see that $x_1=\tfrac{1}{2}b_1$ and $x'$ is half-integer by induction.

So we may assume that $D$ is a $0$--$1$ matrix and that all columns and rows of $D$ contain exactly two entries that equal $1$. Thus $D$ is the vertex-edge incidence matrix of a graph $G$ that is the disjoint union of (odd) cycles, and $x$ is the unique perfect fractional $b$-matching in $G$. The half-integrality of $x$ is then obvious. \end{proof}
\end{proof}


Let $S\subseteq T$ be flats in $M$. The sum $\frac{1}{2}(S+T)$ is called a \emph{cover of $(M,L)$} if $a(S+T)_l\geq 2$ for every line $l\in L$. The cover is called \emph{mimimum} if $r(\frac{1}{2}(S+T))=\nu^*(M,L)$. So the minimum covers are the half-integral optimal solutions $y$ with $\supp (y)$ a chain, of the dual program (\ref{dual}) for $w=\one$. 

When $\frac{1}{2}(S+T)$ and $\frac{1}{2}(S'+T')$ are two minimum covers, also $\frac{1}{2}(S\meet S'+T\join T')$ is a minimum cover (see \cite{FMM-duality}). It follows that there is a `canonical' minimum cover $\frac{1}{2}(S^*+T^*)$ with the property that $S^*\subseteq S\subseteq T\subseteq T^*$ for every minimum cover $\frac{1}{2}(S+T)$. This cover is called the \emph{dominant cover} in \cite{FMM-duality}, where a characterization of the dominant cover was given in terms of the maximum size fractional matchings. The \emph{closure} of a fractional matching $x$, denoted by $cl(x)$, is defined to be the smallest flat containing the lines in the support of $x$:
\begin{equation}
\cl(x):=\cl\big(\bigcup\{l\mid x_l>0\}\big).
\end{equation}

\begin{theorem}\label{dominant} \cite{FMM-duality} Let $\frac{1}{2}(S^*+T^*)$ be the dominant cover of $(M,L)$. Then $T^*=\bigcap_x \cl(x)$, where $x$ runs over all  maximum size fractional matchings in $(M,L)$.
\end{theorem}
\begin{theorem}\label{closure} \cite{FMM-duality}
If $x$ is a vertex of the fractional matroid matching polytope, then $cl(x)$ is a tight flat with respect to $x$, that is, $|x|=r(cl(x))$ holds.
\end{theorem}
The characterization of $T^*$ in Theorem \ref{dominant} implies that 
\begin{equation}
S^*=\cl(\bigcup_{l\not\subseteq T^*}(T^*\cap l),
\end{equation}
since $S^*\subseteq T^*$ and $\frac{1}{2}(S^*+T^*)$ must cover each line $l\in L$.

In this paper, we use the cardinality algorithm to construct an algorithm that finds, for given weights on the lines, a maximum weight fractional (perfect) matching and an optimal dual solution.

\subsection*{Oracles and technical remarks}
Chang et al. \cite{FMM-algorithm} presented an algorithm that given the matroid $M$ and the set of lines $L$, computes the dominant cover and a maximum size (extreme) fractional matching $x$ in time polynomial in $|L|$ and $r(E)$. To accomodate infinite matroids (of finite rank), in particular full linear matroids, we need to say how the matroid and the lines are represented. This is left implicit in \cite{FMM-algorithm}.

The matroid $M$ will be given by a rank-oracle. However, the ground set $E$ and the lines $L$ are not given explicitly. Rather, $L$ is indexed by a finite set and we assume an oracle that takes as input an index representing a line $l$ and a flat $F$ represented by a base of $F$, and outputs a base for $l\cap F$. This will suffice for the algorithmic operations in \cite{FMM-algorithm}.

In this paper, we use the cardinality algorithm as a subroutine to find a maximum weight (perfect) fractional matching and an optimal dual solution.
Our algorithm uses direct sums of minors of the matroid $M$. To make the above oracle available in those matroids, we will assume that our oracle, in addition to a base for $l\cap F$ also returns a base for $l\setminus F$. In particular, every element $e\in E$ that occurs during the algorithm is generated this way, starting from the empty flat. 
\\ \\
In addition to lines of rank $2$, we allow for lines of rank $1$, which greatly simplifies notation and arguments in the contracted matroids. We believe that all results from \cite{FMM-algorithm, FMM-duality} are easily extended to this more general setting. Alternatively, the general case can be reduced to the one allowing only lines of rank $2$, as follows.

Denote by $L_i\subseteq L$ the set of lines of rank $i=1,2$. Make a disjoint copy $M'$ of $M$ and a disjoint copy $L'$ of $L$. We denote for every $l\in L$ its copy by $l'$, and similarly denote for every $S\subseteq E$ the set of copies of its elements by $S'$. Let $N=M\bigoplus M'$ be the direct sum of $M$ and $M'$ and define the following set of rank $2$ lines in $N$: $K:=L_2\cup L'_2\cup \{l\cup l'\mid l\in L_1\}$. 

It is easy to check that every fractional matching $x\in \mathbb{R}^L$ yields a fractional matching $z\in \mathbb{R}^K$ twice the size, given by 
$z_l:=z_l':=x_l$ for every $l\in L_2$ and $z_{l\cup l'}:=2x_l$ for every $l\in L_1$. Conversely, every fractional matching $z\in \mathbb{R}^K$ yields a fractional matching $x\in \mathbb{R}^L$ of (at least) half the size of $z$.

Similarly, every cover $\tfrac{1}{2}(S+T)$ for $L$ gives a cover $\tfrac{1}{2}((S\cup S')+(T\cup T'))$ of double cost for $K$, and every cover for $K$ gives a cover for $L$ of (at most) half the cost. It is easy to see that the dominant cover for $K$ is `symmetric' and obtained from the dominant cover for $L$. 

Hence finding a maximum size (extreme) fractional matching and dominant cover for $L$ can be reduced to finding such for $K$.

\section{A matroid operation}
Let $M$ be a matroid on $E$. For a flat $F\subseteq E$, we denote by $M|F$ the restriction of $M$ to $F$ and by $M/F$ the  matroid obtained from $M$ by contracting $F$. 
Given a chain $\mathcal{F}=\{F_1,\dots,F_k\}$ of flats with $F_1\subset F_2\cdots\subset F_k$, we define the matroid $M\contract \mathcal{F}$ on the same ground set $E$ by:
\begin{equation}
M\contract\mathcal{F}:=\bigoplus_{i=0}^k (M|F_{i+1})/F_{i},
\end{equation}
where we set $F_0:=\emptyset$ and $F_{k+1}:=E$.
Observe that when $M$ has no loops (as will always be assumed), also $M\contract \mathcal{F}$ has no loops. The flats of $M\contract \mathcal{F}$ are precisely the sets $S_0\cup\cdots\cup S_k$, with $S_i\subseteq F_{i+1}\setminus F_i$ and $S_i\cup F_i$ a flat of $M$ for every $i=0,\ldots, k$. 

For $X\subseteq E$, we can write $X=X_0\cup\cdots\cup X_k$, where $X_i:=X\cap (F_{i+1}\setminus F_i)$ and $F_0:=\emptyset$ and $F_{k+1}:=E$. Using the submodularity of the rank function, we see that
\begin{eqnarray}\label{rankdecreases}
r_{M\contract \mathcal{F}} (X)&=&\sum_{i=0}^k r_{M\contract\mathcal{F}}(X_i)\nonumber\\
&=&\sum_{i=0}^k (r(X_i\cup F_i)-r(F_i))\nonumber\\
&\leq &\sum_{i=0}^k (r(X\cap F_{i+1})-r(X\cap F_i))\nonumber\\
&=&r(X).
\end{eqnarray}

We denote the \emph{closure} of $X\subseteq E$ in the matroid $M\contract \mathcal{F}$ by $X\contract\mathcal{F}$: 
\begin{equation}
X\contract\mathcal{F}:=\cl_{M\contract \mathcal{F}}(X)=\bigcup_{i=0}^k \left( \cl_M((X\cap F_{i+1})\cup F_i)\setminus F_i \right) .
\end{equation}

It is useful to observe that when $\mathcal{F}=\mathcal{F}_1\cup\mathcal{F}_2$, the elements of $\mathcal{F}_2$ are flats in $M\contract\mathcal{F}_1$ and the equality $M\contract\mathcal{F}=(M\contract\mathcal{F}_1)\contract\mathcal{F}_2$ holds. In general, for a subset $X\subseteq E$ we only have the inclusion $X\contract\mathcal{F}\subseteq (X\contract\mathcal{F}_1)\contract\mathcal{F}_2$. 

Since no loops are created when constructing $M\contract \mathcal{F}$ from $M$ it follows from (\ref{rankdecreases}) that every line in $M$ is a line in $M\contract{F}$ as well. We should stress that the degree vector $a(X)$ of $X\subseteq E$ does not depend on the matroid and is the same for $M$ as for $M\contract \mathcal{F}$.

\begin{lemma}\label{contract1}
Let $S=S_0\cup\cdots\cup S_k$ be a flat in $M\contract \mathcal{F}$, with $S_i\subseteq F_{i+1}\setminus F_i$. Then 
\begin{equation}
a(S)=\sum_{i=0}^k \left[a(S_i\cup F_i)-a(F_i)\right].
\end{equation}
\end{lemma}
\begin{proof}
For $i=0,\ldots,k$ let $T_i$ be the flat in $M$ defined by $T_i:=S_i\cup F_i$.
Let $l\in L$. Observe that $l$ intersects at most two of the $S_i$ and if $l$ intersects $S_i$ and $S_j$ for some $i\neq j$, then $l\subseteq S_i\cup S_j$. It follows that $a(S)_l=\sum _i a(S_i)_l$.

It now suffices to show that for every $i=0,\ldots,k$ the equality $a(S_i)_l+a(F_i)_l=a(T_i)_l$ holds. Since $S_i$ and $F_i$ are disjoint, the cases where $a(S_i)_l\in \{0,2\}$ or $a(F_i)\in \{0,2\}$ are clear. In the remaining case $l\cap S_i\not=\emptyset$ and $l\cap F_i\not=\emptyset$. Since $F_i$ is a flat in $M$, this implies that $r(l\cap T_i)>1$. This implies $l\subseteq T_i$, since $r(l)\leq 2$ and $T_i$ is a flat in $M$. 
\end{proof}

The following proposition relates fractional matchings in $M$ to those in $M\contract{\mathcal{F}}$.

\begin{proposition}\label{liftx} If $x$ is a fractional matching in $(M\contract \mathcal{F},L)$, then $x$ is a fractional matching in $(M,L)$. If $x$ is a fractional matching in $(M,L)$, and $a(F)x=r_M(F)$ for every $F\in \mathcal{F}$, then $x$ is a fractional matching in $(M\contract\mathcal{F},L)$.
\end{proposition}

\begin{proof}
To prove the first statement, let $T$ be a flat of $M$. Then 
\begin{equation}
a(T)x\leq a(T\contract\mathcal{F})x\leq r_{M\contract\mathcal{F}}(T\contract\mathcal{F})=r_{M\contract\mathcal{F}}(T)\leq r_M(T),
\end{equation}
where the last inequality follows from (\ref{rankdecreases}).

To prove the second statement, let $S=S_0\cup\cdots\cup S_k$ be a flat in $M\contract \mathcal{F}$. Then 
\begin{eqnarray}
a(S)x&=&\sum_{i=0}^k \left[a(S_i\cup F_i)-a(F_i)\right]x\nonumber\\
&=&\sum_{i=0}^k \left[a(S_i\cup F_i)x-r_M(F_i)\right]\nonumber\\
&\leq&\sum_{i=0}^k \left[r_M(S_i\cup F_i)-r_M(F_i)\right]\nonumber\\
&=&r_{M\contract\mathcal{F}}(S).
\end{eqnarray}
Here the first equality follows from Lemma \ref{contract1} and the second equality is the assumption made in the proposition.
\end{proof}

\section{Algorithm for maximum weight perfect fractional matching}
Let $M$ be a matroid with ground set $E$ and let $L$ be a finite set of lines in $M$. Let $w:L\to \Q_+$ be nonnegative rational weights on the lines. We will assume that $(M,L)$ has a perfect fractional matching. 
Our goal now is to solve the following LP over the perfect fractional matching polytope:
\begin{eqnarray}\label{perfectprimal}
\text{maximize}&&w\cdot x\nonumber\\
\text{subject to}&&x\geq 0\nonumber\\
&&a(T)\cdot x\leq r(T)\quad \text{for all $T\in\mathcal{L}(M), T\ne E$}\nonumber\\
&&a(E)\cdot x= r(E).
\end{eqnarray}
The dual linear program is given over vectors $y\in\R ^{\mathcal{L}(M)}$ of finite support by
\begin{eqnarray}\label{perfectdual}
\text{minimize}&& r(y)\nonumber\\
\text{subject to}&&y(T)\geq 0 \quad \text{for all $T\in\mathcal{L}(M), T\ne E$}\nonumber\\
&&a(y)\geq w.
\end{eqnarray}

Throughout the algorithm we keep a chain $\mathcal{F}=\{F_1,\ldots,F_k\}$ of flats in $M$, with $\emptyset\subset F_1\subset\cdots\subset F_k\subset E$. For convenience we always define $F_0:=\emptyset$ and $F_{k+1}:=E$. We also keep a dual feasible solution $y$ with support $\mathcal{F}$ or $\mathcal{F}\cup\{E\}$. That is $y=\lambda_1F_1+\cdots+\lambda_kF_k+\lambda E$ satisfies $a(y)\geq w$, $\lambda_1,\ldots,\lambda_k\in \Q_{>0}$ and $\lambda\in \Q$.

Observe that for a perfect fractional matching $x$ in $(M,L)$ we have:
\begin{eqnarray}
wx=\sum_{l\in L}w_lx_l&\leq &\sum_{l\in L}a(y)_lx_l\nonumber\\
&=&\sum_{i=1}^k \lambda_i a(F_i)x + \lambda a(E)x\nonumber\\
&\leq &\sum_{i=1}^k \lambda_i r(F_i) +\lambda r(E)=r(y). 
\end{eqnarray}
Here the equality $a(E)x=r(E)$ is used in the last line. It follows that $x$ has maximum weight if and only if
\begin{eqnarray}
w_l=a(y)_l &\text{ for all }& l\in \supp(x),\label{slack1}\\
a(F)x=r(F) &\text{ for all }& F\in \supp(y).\label{slack2}
\end{eqnarray}
Denote $L_y:=\{l\in L\mid w_l=a(y)_l\}$. The above two conditions will be met once we find a perfect fractional matching $x$ in $(M\contract \mathcal{F},L_y)$. Indeed, in that case $x$ is also a perfect fractional matching in $(M,L)$ by Proposition \ref{liftx}. Equalities (\ref{slack1}) will be satisfied by definition of $L_y$, and equalities (\ref{slack2}) will be satisfied because $a(E)x=r(E)$ implies that for every $i$ 
\begin{eqnarray}
r(E)=a(E)x&=& a(F_i)x+a(E\setminus F_i)x\nonumber\\
&\leq&r_{M\contract\mathcal{F}}(F_i)+r_{M\contract\mathcal{F}}(E\setminus F_i)\nonumber\\
&=&r(F_i)+(r(E)-r(F_i))\nonumber\\
&=&r(E),
\end{eqnarray}
and hence $a(F_i)x=r(F_i)$. The description of the algorithm is as follows. 

Initially, we set $\mathcal{F}:=\emptyset$, $y:=\lambda E$, where $\lambda:=\frac{1}{2}\max \{w_l \mid l\in L\}$. The iteration of the algorithm is as follows.

Find a maximum size fractional matching $x$ in $(M\contract \mathcal{F},L_y)$, and find the dominant cover $\frac{1}{2}(S^*+T^*)$, of $(M\contract \mathcal{F},L_y)$. This can be done by the algorithm from \cite{FMM-algorithm}. Now consider two cases.

\textbf{Case 1:} If $x$ is a \emph{perfect} fractional matching, $x$ will be a maximum weight perfect fractional matching in $(M,L)$ and $y$ is an optimal dual solution. We output $x$ and $y$ and stop. 

\textbf{Case 2:} If $x$ is not perfect, write $S^*=S^*_0\cup\cdots \cup S^*_k$ and $T^*=T^*_0\cup\cdots\cup T^*_k$, where $S^*_i,T^*_i\subseteq F_{i+1}\setminus F_i$ for $i=0,\ldots,k$. We define 
\begin{equation}
z:=\sum_{i=0}^k (S^*_i\cup F_i-F_i)+\sum_{i=0}^k (T^*_i\cup F_i-F_i)- E.
\end{equation}
Observe that $a(z)_l=a(S^*+T^*)_l-2\in \{-2,-1,0,1,2\}$ for every $l\in L$. Furthermore, since $\tfrac{1}{2}(S^*+T^*)$ is a cover in $(M\contract\mathcal{F}, L_y)$ we have $a(z)_l\geq 0$ for $l\in L_y$ and $a(z)_l=0$ for $l\in \supp (x)$ by complementary slackness.

Let 
\begin{eqnarray}
\epsilon_1&:=& \max \{t\geq 0 \mid (y+tz)_F\geq 0 \text{ for all $F\not= E$}\},\nonumber\\
\epsilon_2&:=& \max \{t\geq 0 \mid a(y+tz)\geq w\},
\end{eqnarray}
and let $\epsilon$ be the minimum of $\epsilon_1$ and $\epsilon_2$. Let $y':=y+\epsilon z$ and $\mathcal{F}':=\supp(y')\setminus \{E\}$. By definition of $\epsilon$, $y'$ will be a dual feasible solution. We reset $y:=y'$ and $\mathcal{F}:=\mathcal{F}'$ and iterate. 

In Section \ref{RunningTime} we will show that the algorithm terminates in $O(r(E)^3)$ iterations.

\subsection*{Finding a maximum weight fractional matching}
The algorithm for finding a maximum weight \emph{perfect} fractional matching described above, can be used to find a maximum weight fractional matching as follows.

Let the matroid $M$, a set of lines $L$ and weights $w:L\to \Q$ be given. We may assume that $w$ is nonnegative since we can remove lines of negative weight. Let $B$ be a base of $M$. By introducing parallel copies if necessary, we may assume that $\{b\}\not\in L$ for every $b\in B$. We define $L':=L\cup \{\{b\}\mid b\in B\}$ and $w':L'\to \N$ by $w'(l):=w(l)$ for $l\in L$ and $w'(\{b\}):=0$ for $b\in B$.  

Let $x$ be a vertex of the fractional matroid matching polytope for $(M,L)$, and $B'\subset B$ be a base of $M/\cl(x)$. Define $x':L'\to \{0,\frac{1}{2},1\}$ by $x'_l:=x_l$ for $l\in L$, $x'_{\{b\}}:=\frac{1}{2}$ for $b\in B'$ and $x'_{\{b\}}:=0$ for $b\in B\setminus B'$. Then $x'$ is a fractional matching in $(M,L')$ since for any flat $F$ we have
\begin{eqnarray}
a_F\cdot x'&=&a_{F\wedge \cl(x)}\cdot x+|F\cap B'|\nonumber\\
&\leq &r(F\wedge \cl(x))+|F\cap B'|\nonumber\\
&\leq&r(F).
\end{eqnarray}
For $F=E$, both inequalities are satisfied with equality: the first by Theorem \ref{closure} and the second by definition of $B'$. Hence $x'$ is a perfect fractional matching in $(M,L')$. 

It follows that the fractional matchings in $(M,L)$ are precisely the restrictions to $L$, of the perfect fractional matchings in $(M,L')$. Using the above algorithm we find a maximum $w'$-weight perfect fractional matching $x'$ in $(M,L')$ and an optimal dual solution $y'$. Restricting $x'$ to $L$ gives the required maximum $w$-weight fractional matching in $(M,L)$. Since $a(y')_{\{b\}}\geq w'(\{b\})=0$ for every $b\in B$, it follows from the fact that $\supp (y')$ is a chain, that $y'_E\geq 0$ and hence is an optimal dual solution for the problem of finding a maximum weight fractional matching in $(M,L)$.

\section{Running time of the algorithm}\label{RunningTime}
In this section, we will show that the algorithm for maximum weight perfect fractional matching described in the previous section terminates after $O(r^3)$ iterations. To this end, we will define a parameter $\psi(\mathcal{F}, S^*,T^*)$ that measures the progress made.

Let $\mathcal{F}=\{F_1, \ldots, F_k\}$ be a chain of flats in $M$, where $\emptyset=:F_0\subset F_1\subset\cdots\subset F_k\subset F_{k+1}:=E$. For a subset $X=X_0\cup\cdots\cup X_k$ of $E$ where $X_i\subseteq F_{i+1}\setminus F_i$ ($i=0,\ldots,k$), we define
\begin{equation}
\phi(\mathcal{F},X):=\sum_{i=0}^k \left[r(F_i\cup X_i)^2-r(F_i)^2\right].
\end{equation}

Observe that $\phi(\mathcal{F},X)$ is a nonnegative integer no larger than $r(E)^2$. Clearly, for $X\subseteq X'$ we have $\phi(\mathcal{F},X)\leq \phi(\mathcal{F},X')$ and $\phi(\mathcal{F},X)=\phi(\mathcal{F},X\contract\mathcal{F})$.

For flats $S$ and $T$ of $M$, we define
\begin{equation}
\psi(\mathcal{F},S,T):=\phi(\mathcal{F},S)+\phi(\mathcal{F},T)+2r(E)r_{M\contract \mathcal{F}}(T).
\end{equation}
So $\psi(\mathcal{F},S,T)$ is a nonnegative integer no larger than $4r(E)^2$.
The following two lemmas will display some useful properties of $\phi$ (and hence of $\psi$).

\begin{lemma}\label{psi1}
Let $X\subseteq X'$ be subsets of $E$. Then 
\begin{equation}
\phi(\mathcal{F},X')-\phi(\mathcal{F},X)\leq (r_{M\contract\mathcal{F}}(X')-r_{M\contract\mathcal{F}}(X))\cdot 2r(E),
\end{equation}
and equality holds if and only if $r_{M\contract\mathcal{F}}(X')=r_{M\contract\mathcal{F}}(X)$.
\end{lemma}
\begin{proof}
Write $X=X_0\cup\cdots\cup X_k$ and $X'=X'_0\cup\cdots\cup X'_k$ where $X_i\subseteq X'_i\subseteq F_{i+1}\setminus F_i$. We have
\begin{eqnarray}
\phi(\mathcal{F},X')-\phi(\mathcal{F},X)&=&\sum_{i=0}^k (r(F_i\cup X'_i)^2-r(F_i\cup X_i)^2)\nonumber\\
&=&\sum_{i=0}^k (r(F_i\cup X'_i)-r(F_i\cup X_i))(r(F_i\cup X'_i)+r(F_i\cup X_i))\nonumber\\
&\leq&\sum_{i=0}^k (r(F_i\cup X'_i)-r(F_i\cup X_i))\cdot 2r(E)\nonumber\\
&=&(r_{M\contract\mathcal{F}}(X')-r_{M\contract\mathcal{F}}(X))\cdot 2r(E).
\end{eqnarray}
Equality holds if and only if for every $i$ we have $r(F_i\cup X'_i)=r(F_i\cup X_i)$.
\end{proof}

\begin{lemma}\label{psi2}
Let $X=X_0\cup\cdots\cup X_k$ with $X_i\subseteq F_{i+1}\setminus F_i$, and let $\mathcal{F}'\supset \mathcal{F}$ be a longer chain of flats in $M$. Suppose that $r_{M\contract \mathcal{F}}(X)=r_{M\contract\mathcal{F'}}(X)$. Then $\phi(\mathcal{F},X)\leq \phi(\mathcal{F'},X)$, and equality holds if and only if $\mathcal{F'}\cup\{F_i\join X_i, i=0,\ldots, k\}$ is a chain.
\end{lemma}
\begin{proof}
Consider a fixed $i\in \{0,\ldots, k\}$. Let $H_0\subset\cdots\subset H_{l+1}$ be the flats $\{F\in \mathcal{F'}\mid F_{i}\subseteq F\subseteq F_{i+1}\}$. Write $X_i:=Y_0\cup\cdots\cup Y_l$ where $Y_j\subseteq H_{j+1}\setminus H_j$ for $j=0,\ldots,l$. Denote $r:=r(F_i\cup X_i)-r(F_i)$ and $r_j:=r(H_j\cup Y_j)-r(H_j)\leq r(H_{j+1})-r(H_j)$ for $j=0,\ldots,l$.

The assumpion $r_{M\contract \mathcal{F}}(X)=r_{M\contract \mathcal{F'}}(X)$ means that $r=r_0+\cdots+r_l$. We have the following inequalities.
\begin{eqnarray}
r(F_i\cup X_i)^2-r(F_i)^2&=&(r(F_i)+r)^2-r(F_i)^2\nonumber\\
&=& \sum_{j=0}^l \left[ (r(F_i)+r_0+\cdots+r_{j-1}+r_j)^2-(r(F_i)+r_0+\cdots+r_{j-1})^2\right]\nonumber\\
&\leq &\sum_{j=0}^l \left[ (r(H_j)+r_j)^2-r(H_j)^2\right]\nonumber\\
&=&\sum_{j=0}^l (r(H_j\cup Y_j)^2-r(H_j)^2).
\end{eqnarray}
The inequality follows from the fact that $r(H_j)\geq r(F_i)+r_0+\cdots+r_{j-1}$ for $j=0,\ldots,l$, and equality holds if and only if for every $j$ either $r(H_j)=r(F_i)+r_0+\cdots+r_{j-1}$ or $r_j=0$. That is, if $j$ is the largest index for which $r_j>0$, then we have equality if and only if $H_j\subseteq F_i\join X_i\subseteq H_{j+1}$. 
Summing over all $i$ proves the lemma.
\end{proof}

In the remainder of this section, we shall prove that the algorithm described in the previous section is correct, and that the number of iterations is at most $O(r(E)^3)$. 

\begin{proof}[Proof of the running time bound]
We will show that the pair of nonnegative integers $(r(E)-2\nu^*(M\contract\mathcal{F},L_y),\ \psi(\mathcal{F},S^*,T^*))$ decreases lexicographically in each iteration. 

Consider an iteration. Let $y$ be the current dual solution, $\mathcal{F}$ the associated chain of flats in $M$ and let $\frac{1}{2}(S^*+T^*)$ be the dominant cover of $(M\contract \mathcal{F}, L_y)$. Denote by $y'$, $\mathcal{F}'$ and $\frac{1}{2}({S^*}'+{T^*}')$ the corresponding objects in the next iteration.

We define the subset $\overline{L}$ of lines by   
\begin{eqnarray}
\overline{L}&:=&\{l\in L_y\mid l\in \supp(x)\nonumber\\
&&\text{ for some maximum size fractional matching $x$ in $(M\contract\mathcal{F},L_y)$ }\},
\end{eqnarray}
and the chain $\overline{\mathcal{F}}$ of flats in $M$ by
\begin{equation}
\overline{\mathcal{F}}:=\mathcal{F}\cup \{S^*_i\cup F_i, i=0,\ldots, k\} \cup \{T^*_i\cup F_i, i=0,\ldots, k\},
\end{equation}
where again we denote $S^*_i:=S^*\cap (F_{i+1}\setminus F_i)$ and $T^*_i:=T^*\cap (F_{i+1}\setminus F_i)$ for $i=0,\ldots, k$. Note that $\overline{\mathcal{F}}=\mathcal{F}\cup\mathcal{F}'$, and hence $M\contract \overline{\mathcal{F}}=(M\contract \mathcal{F})\contract \mathcal{F}'=(M\contract\mathcal{F}')\contract \mathcal{F}$.

By complementary slackness, $a(\frac{1}{2}(S^*+T^*))_l=2$ holds for every $l\in \overline{L}$. Therefore we have $a(z)_l=0$ for every $l\in \overline{L}$ so that 
\begin{equation}
\overline{L}\subseteq L_{y'}.
\end{equation}

Again by complementary slackness, $a(S^*)x=r_{M\contract\mathcal{F}}(S^*)$ and $a(T^*)x=r_{M\contract\mathcal{F}}(T^*)$ holds for any maximum size fractional matching $x$ in $(M\contract\mathcal{F}, L_y)$. Hence the flats $F_i\cup S_i^*$ and $F_i\cup T_i^*$ are tight for $x$, which implies by Proposition \ref{liftx} that 
$x$ is also a fractional matching in $(M\contract \overline{\mathcal{F}},\overline{L})$ and hence in $(M\contract \mathcal{F}',L_{y'})$. We can conclude that  
\begin{equation}\label{nuincrease}
\nu^*(M\contract \mathcal{F},L_y)=\nu^*(M\contract \overline{\mathcal{F}},\overline{L})\leq \nu^*(M\contract \mathcal{F'},L_{y'}).
\end{equation}
When strict inequality holds in (\ref{nuincrease}), we are done. Therefore, in the remainder of the proof, we may assume that 
\begin{equation}
\nu^*(M\contract \mathcal{F},L_y)=\nu^*(M\contract \mathcal{F'},L_{y'}).
\end{equation}
Denote 
\begin{equation}
\overline{S}:={S^*}'\contract \overline{\mathcal{F}},\quad \overline{T}:={T^*}'\contract \overline{\mathcal{F}}.
\end{equation}

Clearly $\tfrac{1}{2}(\overline{S}+\overline{T})$ is a cover in $(M\contract\overline{\mathcal{F}},\overline{L})$, since $\overline{L}\subseteq L_{y'}$. Since $\nu^*(M\contract\overline{\mathcal{F}},\overline{L})=\nu^*(M\contract\mathcal{F}',L_{y'})$ and 
\begin{equation}
r_{M\contract{\overline{\mathcal{F}}}}(\overline{S}+\overline{T})=r_{M\contract{\overline{\mathcal{F}}}}({S^*}'+{T^*}')\leq r_{M\contract \mathcal{F}'}({S^*}'+{T^*}'),
\end{equation}
we find that $\frac{1}{2}(\overline{S}+\overline{T})$ is a minimum cover in $(M\contract\overline{\mathcal{F}},\overline{L})$. This implies that $r_{M\contract{\overline{\mathcal{F}}}}({S^*}')= r_{M\contract \mathcal{F}'}({S^*}')$ and $r_{M\contract{\overline{\mathcal{F}}}}({T^*}')= r_{M\contract \mathcal{F}'}({T^*}')$.

Clearly, $\frac{1}{2}(S^*+T^*)$ is also a minimum cover of $(M\contract\overline{\mathcal{F}},\overline{L})$. Furthermore, since $x$ is a maximum size fractional matching in $(M\contract \mathcal{F},L_y)$ if and only if $x$ is a maximum size fractional matching in $(M\contract \overline{\mathcal{F}},\overline{L})$, it follows that $\frac{1}{2}(S^*+T^*)$ is in fact the dominant cover of $(M\contract\overline{\mathcal{F}},\overline{L})$. Indeed, let $\frac{1}{2}(S+T)$, $S\subseteq T$ be the dominant cover of $(M\contract\overline{\mathcal{F}},\overline{L})$. It suffices to show that $T^*=T$. For every maximum size fractional matching $x$ in $(M\contract\mathcal{F},L_y)$, we have $\cl_{M\contract\mathcal{F}}(x)\supseteq T^*$. This implies that $\cl_{M\contract\mathcal{F}}(x)$ is also a flat in $M\contract\overline{\mathcal{F}}$, and hence $\cl_{M\contract\mathcal{F}}(x)=\cl_{M\contract\overline{\mathcal{F}}}(x)$. By the characterization of the dominant cover, Theorem \ref{dominant}, it now follows that
\begin{equation}
T=\bigcap_x \cl_{M\contract\overline{\mathcal{F}}}(x)=\bigcap_x\cl_{M\contract\mathcal{F}}(x)=T^*,
\end{equation} 
where the intersection runs over all maximum size fractional matchings $x$ in $(M\contract\mathcal{F},L_y)$.

By comparing the two covers of $(M\contract \overline{\mathcal{F}},\overline{L})$, we obtain 
\begin{equation}\label{comparecover}
S^*\subseteq \overline{S},\quad \overline{T}\subseteq T^*.
\end{equation} 

We have the following series of inequalities:
\begin{eqnarray}\label{comparepsi}
\psi(\mathcal{F'},{S^*}',{T^*}')&=&\phi(\mathcal{F}',{S^*}')+\phi(\mathcal{F}',{T^*}')+2r(E)r_{M\contract{\mathcal{F}'}}({T^*}')\nonumber\\
&\leq&\phi(\overline{\mathcal{F}},\overline{S})+\phi(\overline{\mathcal{F}},\overline{T})+2r(E)r_{M\contract\overline{\mathcal{F}}}(\overline{T})\label{ineqpart1}\\
&\leq & \phi(\overline{\mathcal{F}},S^*)+\phi(\overline{\mathcal{F}},T^*)+2r(E)r_{M\contract\overline{\mathcal{F}}}(\overline{T}+\overline{S}-S^*)\label{ineqpart2}\\
&=&\phi(\overline{\mathcal{F}},S^*)+\phi(\overline{\mathcal{F}},T^*)+2r(E)r_{M\contract\overline{\mathcal{F}}}(T^*)\label{eqpart1}\\
&=&\phi(\mathcal{F},S^*)+\phi(\mathcal{F},T^*)+2r(E)r_{M\contract\mathcal{F}}(T^*)\label{eqpart2}\\
&=&\psi(\mathcal{F},S^*,T^*).\nonumber
\end{eqnarray}
Here (\ref{ineqpart1}) follows from Lemma \ref{psi2} since $r_{M\contract \mathcal{F}'}({S^*}')= r_{M\contract \overline{\mathcal{F}}}({S^*}')$ and $r_{M\contract \mathcal{F}'}({T^*}')= r_{M\contract \overline{\mathcal{F}}}({T^*}')$. The second inequality (\ref{ineqpart2}) follows from Lemma \ref{psi1} applied to $S^*\subseteq \overline{S}$ and the fact that $T^*\supseteq \overline{T}$. Equality in (\ref{eqpart1}) follows since $r(S^*+T^*)=r(\overline{S}+\overline{T})$. Finally, equality in (\ref{eqpart2}) follows by Lemma \ref{psi2} since $r_{M\contract \mathcal{F}}(S^*)= r_{M\contract \overline{\mathcal{F}}}(S^*)$ and $r_{M\contract \mathcal{F}}(T^*)= r_{M\contract \overline{\mathcal{F}}}(T^*)$. 

We need to show that $\psi(\mathcal{F'},{S^*}',{T^*}')<\psi(\mathcal{F},S^*,T^*)$ holds. Suppose that equality holds in (\ref{ineqpart2}). Then we must have equality in (\ref{comparecover}), by Lemma \ref{psi1}. This means that we are in the case $\epsilon_1<\epsilon_2$. Indeed, if $\epsilon_1\geq \epsilon_2$, there exists an $l\in L_{y'}\setminus L_y$ with $a(S^*+T^*)_l<2$. However, $a(\overline{S}+\overline{T})_l=2$, which would imply that we do not have equality in (\ref{comparecover}).

Since $\epsilon_1<\epsilon_2$, there is an $F_i\in \mathcal{F}\setminus \mathcal{F'}$.  Since $z_{F_i}<0$, we have either ($S^*_i\not=\emptyset$ and $S^*_{i-1}\not\supseteq F_i\setminus F_{i-1}$) or ($T^*_i\not=\emptyset$ and $T^*_{i-1}\not\supseteq F_i\setminus F_{i-1}$). 

Now let $\mathcal{F}'=\{F'_1,F'_2,\cdots ,F'_m\}$, and choose $j$ such that $F'_j\subsetneq F_i\subsetneq F'_{j+1}$.

First, assume that $S^*_i\not=\emptyset$ and $S^*_{i-1}\not\supseteq F_i\setminus F_{i-1}$. Then we apply Lemma \ref{psi2} with respect to $\mathcal{F}'$ and $\overline{\mathcal{F}}$, and $X:=S^*=\overline{S}$. Note that $F'_j\join X_j=F'_j\join \overline{S}_i$. Since $\{F'_j\join \overline{S}_i\}\cup \overline{\mathcal{F}}$ is not a chain, we get that $\phi (\mathcal{F}',{S^*}')<\phi (\overline{\mathcal{F}},\overline{S})$, and thus, $\psi(\mathcal{F'},{S^*}',{T^*}')<\psi(\mathcal{F},S^*,T^*)$. 

Second, assume that $T^*_i\not=\emptyset$ and $T^*_{i-1}\not\supseteq F_i\setminus F_{i-1}$. Here, if $S^*_i\neq\emptyset$, then the first case applies, thus we may also assume that $S^*_i=\emptyset$. Apply Lemma \ref{psi2} with respect to $\mathcal{F}'$ and $\overline{\mathcal{F}}$, and $X:=T^*=\overline{T}$. Note that $F'_j\join X_j=F'_j\join \overline{T}_i$. Since $\{F'_j\join \overline{T}_i\}\cup \overline{\mathcal{F}}$ is not a chain, we get that $\phi (\mathcal{F}',{T^*}')<\phi (\overline{\mathcal{F}},\overline{T})$, and thus, $\psi(\mathcal{F'},{S^*}',{T^*}')<\psi(\mathcal{F},S^*,T^*)$. This completes the proof.

\end{proof}


\section*{Acknowledgements}
We thank the anonymous referee for the very useful comments. The first author would like to thank Rudi Pendavingh for stimulating discussions on a possible weighted version of the fractional matching algorithm for matroids.

\end{document}